\documentclass[12pt]{article}
\usepackage{amsmath,amsthm,amsfonts,amssymb,amscd,enumerate}
\RequirePackage{graphics}

 \theoremstyle{plain} \newtheorem{theorem}{Theorem}[section]

\newtheorem{lemma}[theorem]{Lemma}

 \theoremstyle{definition}
\newtheorem{definition}[theorem]{Definition} \theoremstyle{remark}
\newtheorem{remark}[theorem]{Remark}

\newcommand{\eltwo}{\ell^2} 
\newcommand{\ca}{\mathcal {A}} 
 
 \newcommand{\ch}{\mathcal {H}}
 
\newcommand{\cn}{\mathcal {N}} \newcommand{\ct}{\mathcal {S}}
 \newcommand{\cT}{\mathcal {T}}
\newcommand{\cu}{\mathcal {U}} 
 
\newcommand{\cc}{{\mathbb C}} 
 
\newcommand{\rr}{{\mathbb R}} 
\newcommand{\zz}{{\mathbb Z}}

\newcommand{\sg}{\mathrm {sing}}

\newcommand{\hchi}{H_{\chi,a}}
\newcommand{\mt}{\Cal R_X}

\newcommand{\hred}{H_{\mathrm{red}}}

\newcommand{\ol}{\overline} \newcommand{\wt}{\widetilde}
\newcommand{\OL}{\ol{L}}
\newcommand{\wh}{\widehat}
\newcommand{\ga}{\alpha} \newcommand{\gb}{\beta}

\newcommand{\gs}{\sigma} 
 
\newcommand{\gl}{\lambda}

 \newcommand{\gD}{\Delta}
 \newcommand{\gL}{\Lambda}
\newcommand{\gS}{\Sigma}

 \newcommand{\wu}{\widehat{\cu}}

\newcommand{\Card}{\operatorname{Card}}
\newcommand{\Hom}{\operatorname{Hom}}

\newcommand{\Ker}{\operatorname{Ker}}

\newcommand{\Min}{\operatorname{Min}}
\newcommand{\GR}{\operatorname{Gr}}
\newcommand{\MU}{\Min (U)}
\newcommand{\rk}{\operatorname{rk}}

\newcommand{\tc}{{T^{\mathrm{cpt}}}}
\newcommand{\ctc}{{\cT^{\mathrm{cpt}}_X}}
\newcommand{\gsc}{{\gS^{\mathrm{cpt}}}}
\newcommand{\tgsc}{{\wt{\gS}^{\mathrm{cpt}}}}
\newcommand{\tH}{{H^{\mathrm{cpt}}}}

\newcommand{\Cal}[1]{\mathcal{#1}}
\newcommand{\rx}{\mathcal{R}_X}
\newcommand{\C}[0]{\mathbb{C}}
\newcommand{\R}[0]{\mathbb{R}}

\newenvironment{enumeratei}{\begin{enumerate}[\upshape (i)]}%
		{\end{enumerate}}

\begin{document}

\title{Vanishing results for the cohomology of complex toric hyperplane complements}

\author{M. W. Davis \thanks{The first author was partially supported by
NSF grant DMS 1007068.}  \and
S. Settepanella \thanks{The second author was partially supported by the Institute for New Economic Thinking, INET inaugural grant $\sharp$220.} }
\date{\today} \maketitle

\begin{abstract}

Suppose $\Cal R$ is the complement of an essential arrangement of toric hyperlanes in the complex torus $(\C^*)^n$ and $\pi=\pi_1(\Cal R)$.  We show that $H^*(\Cal R;A)$ vanishes except in the top degree $n$ when $A$ is one of the following systems of local coefficients: (a) a system of nonresonant coefficients in a complex line bundle, (b) the von Neumann algebra $\cn\pi$, or (c) the group ring $\zz \pi$.  In case (a) the dimension of $H^n$ is $|e(\Cal R)|$ where $e(\Cal R)$ denotes the Euler characteristic, and in case (b) the $n^{\mathrm{th}}$ $\eltwo$ Betti number is also $|e(\Cal R)|$.
\smallskip

\noindent
\textbf{AMS classification numbers}.  Primary: 52B30
Secondary: 32S22, 52C35, 57N65, 58J22.  
\smallskip

\end{abstract}

\section{Introduction}\label{intro}

A \emph{complex toric arrangement} is a family of complex subtori of a complex torus $(\cc^*)^n$. The study of such objects is a relatively recent topic.
Different versions of these arrangements, also known as toral arrangements, have been introduced and studied in works of Lehrer \cite{lehrer95toral,L2}, Dimca-Lehrer \cite{DimLer}, Douglass \cite{Doug}, Looijenga \cite{Lo} and Macmeikan \cite{Mac1,Mac2}. 

The foundation of the topic can be traced to the paper 
\cite{de2005geometry} by De Concini and Procesi. There the main objects are defined and the cohomology of the complement of a toric arrangement is studied. An explicit goal of \cite{de2005geometry} is to generalize the theory of hyperplane arrangements. (For an extensive account of the work of De Concini and Procesi see \cite{deconcini2010topics}.)

The next step is the work of Moci, in particular
his papers \cite{mocicombinatorics, Moci1,mociwonderful2009}, developing the theory with a
special focus on combinatorics. In  \cite{SimoLuca} Moci and the second author study the homotopy
type of the complement of a special class of toric arrangements which they call \textit{thick}. In \cite{Dantdel} D'Antonio and Delucchi generalize results in \cite{SimoLuca} to a wider class of toric arrangements which they call \textit{complexified} because of structural affinity with the case of hyperplane arrangements. 
They also prove that complements of complexified toric arrangements are minimal (see \cite{Dantdel2}).

In this paper we generalize to toric arrangements  a well known result for affine arrangements:   vanishing conditions for the cohomology of the complement $M(\Cal A)$ of an arrangement $\Cal A$ with coefficients in a complex local system $A$.
Necessary conditions for $H^k(M(\Cal A); A)=0$ if $k \neq n$, i.e., for the cohomology to be concentrated in top dimension, have been  determined by a number of authors, including Kohno \cite{Ko}, Esnault, Schechtman and Viehweg \cite{ESV}, Davis, Januszkiewicz and Leary \cite{djl}, Schechtman, Terao and Varchenko \cite{STV} and Cohen and Orlik  \cite{CoOr}. In particular, in \cite{STV} (see also \cite{CoOr}) it is proved that the cohomology of the complement $M(\Cal A)$ of an arrangement with coefficients in a complex local system  is concentrated in top dimension provided  certain \textit{nonresonance} conditions for monodromies are fulfilled for a certain subset of  edges $($i.e., intersections of hyperplanes$)$ that are called \textit{denses}.

In order to generalize the above results we use techniques developed by the first author in a joint work with Januszkiewicz, Leary and Okun, \cite{djl, djlo, dl, do}. One considers an open
cover of the complement $M$ by ``small'' open sets each homeomorphic to
the complement of a central arrangement.  In the cases of nonresonant rank one local coefficients or $\eltwo$ coefficients, the $E_1$ page of the resulting Mayer-Vietoris spectral sequence is nonzero only along the bottom row, where it can be identified with the simplicial cochains with constant
coefficients on a pair $(N(\cu),N(\cu_\sg))$, which is homotopy equivalent
to $(\cc^n,\gS)$ where $\gS$ is the union of all hyperplanes in the arrangement. 
(The simplicial complex $N(\cu)$ is the nerve of an open cover of $\cc^n$ and $N(\cu_\sg)$ is a subcomplex.)

It follows that the $E_2$ page can be nonzero only in position $(l,0)$. One also can prove that for an affine hyperplane arrangement of rank $l$ only the $l^{\mathrm{th}}$ $\eltwo$-Betti number of the complement $M$ can be nonzero and that it is equal to the rank of the reduced $(l-1)$-homology of $\Sigma$ (cf.~\cite{djl}). Similarly, with coefficients in the group ring, $\zz \pi$, for $\pi=\pi_1(M)$, $H^*(M;\zz\pi)$ is nonzero only in degree $l$ (cf.~\cite{djlo}).
We generalize all three of these vanishing results to the toric case in Theorems \ref{t:generic}, \ref{t:eltwo} and \ref{t:gprg}.

In recent work \cite{PapaSu}, Papadima and Suciu generalize  the result in \cite{CoOr} to arbitrary minimal CW-complex, i.e., a  complex having as many $k$-cells as the $k$-th Betti number. 
It would be very interesting to decide if the complement of toric arrangement also could be minimal. In this case Theorem \ref{t:generic} would be a consequence of minimality. 

Our paper begins with a review of some background about toric and affine arrangements. Then, in Section~ \ref{s:covers}, we give a brief account of open covers by ``small'' convex sets.  
In Section \ref{ss:local} we recall basic definitions on systems of local coefficients. Finally in Section \ref{s:mv} we prove that the cohomology of the complement of a toric arrangement with coefficient in a local system $A$ vanishes except in the top degree when $A$  is a nonresonant local system, the von Neumann algebra $\cn\pi$ or the group ring $\zz \pi$. 

\section{Affine and toric hyperplane arrangements}\label{s:hyper}
\paragraph{Affine hyperplanes arrangements}
A \emph{hyperplane arrangement} $\ca$ is a finite collection of affine
hyperplanes in $\cc^n$.  A \emph{subspace} of $\ca$ is a nonempty
intersection of hyperplanes in $\ca$.  Denote by $L(\ca)$ the poset of
subspaces, partially ordered by inclusion, and let $\ol{L}(\ca):=L(\ca)\cup \{\cc^n\}$.  An arrangement is
\emph{central} if $L(\ca)$ has 
a minimum element.  Given $G\in
L(\ca)$, its \emph{rank}, $\rk(G)$, is the codimension of $G$ in
$\cc^n$.  The minimal elements of $L(\ca)$ form a family of parallel
subspaces and they all have the same rank.  The \emph{rank} of an
arrangement $\ca$ is the rank of a minimal element in $L(\ca)$.  $\ca$
is \emph{essential} if $\rk(\ca)=n$.

The \emph{singular set} $\gS(\ca)$ of $\Cal A$ is the union of
hyperplanes in $\ca$.
The complement of $\gS(\ca)$ in $\C^n$ is denoted $M(\ca)$. 

\paragraph{Toric arrangements}
Let $T=(\C^*)^n$ be a complex torus and let $\gL= Hom(T,\C^*) $ denote the group of characters of $T$. Then $\gL\cong \zz^n$.  A character is \emph{primitive} if it is a primitive vector in $\gL$.  Given a primitive character $\chi$ and an element $a\in \C^*$ put 
\[
\hchi=\{t\in T\mid \chi(t)=a\}.
\]
The subtorus $\hchi$ is a \emph{toric hyperplane}. A finite subset $X \subset \gL \times \C^*$ defines a \emph{toric arrangement},
\[
\cT_X:=\{\hchi\}_{(\chi, a) \in X}
\]
The projection of $X$ onto the first factor is denoted $p(X)$ and is called the \emph{character set} of $\cT_X$.  (Thus, $p(X):=\{\chi\mid (\chi,a)\in X\}$.)  The \emph{singular set}, $\Sigma_X$, is the union of toric hyperplanes in the 
arrangement.  Its complement, $T-\Sigma_X$, is denoted $\rx$. The \textit{intersection poset} $L_X$ is the set of nonempty intersections  
of toric hyperplanes and $\ol{L}_X=L_X \cup \{T\}$. $\ol{L}_X$ is partially ordered by inclusion. The \emph{rank} of the arrangement is the dimension of the linear subspace of $\gL\otimes_\zz \rr$ spanned by $p(X)$.  The arrangement is \emph{essential} if its rank is $n$.

Suppose $G\in L_X$.  Choose a point $x\in G$.  The \emph{tangential arrangement along $G$} is the arrangement $\ca_G$ of linear hyperplanes which are tangent to the complex toric hyperplanes containing $G$ (i.e., all hyperplanes of the form $T_x(H_{\chi,a})$ where $T_x(G)\subset T_x(H_{\chi,a})$).  It is a central hyperplane arrangement of rank equal to $n-\dim G$.

Given a toric arrangement $\cT_X$ of rank $l$, let $K_X$ denote the identity component of the intersection of all kernels in $p(X)$, i.e., $K_X$ is the identity component of 
	\[
	 \bigcap_{\chi \in p(X)} \Ker \chi = \{t\in T\mid \chi(t)=1, \ \forall \chi \in p(X)\}.
	\]
Put $\ol{T}_X:= T/K_X$.  Thus, $K_X$ and $\ol{T}_X$ are tori of dimensions $n-l$ and $l$, respectively.  ($K_X\cong (\C^*)^{n-l}$ and $\ol{T}_X\cong (\C^*)^l$.)  Let $\ol{\Sigma}_X$ denote the image of $\Sigma_X$ in $\ol{T}_X$.  Since $T\to T/K_X$ is a trivial $K_X$-bundle, we have a homeomorphism of pairs,
	\begin{equation}\label{e:tsigma}
	(T,\Sigma_X) \cong K_X \times (\ol{T}_X, \ol{\Sigma}_X).
	\end{equation}
In other words, the arrangement in $T$ is just the product of the arrangement in $\ol{T}_X$ with the torus $K_X$.  We call $\ol{\cT}_X$  the \emph{essentialization} of $\cT_X$. So, it is not restrictive to consider essential toric arrangements.

\begin{lemma}\label{l:conc}
\textup{(cf. \cite[Prop.~2.1]{djl}).}  Suppose $\cT_X$ is an essential toric arrangement on $T$ and $\Sigma= \Sigma_X$.  Then $H_*(T, \Sigma)$ is free abelian and concentrated in degree $n$.
\end{lemma}

\begin{proof}
We follow the ``deletion-restriction'' argument in \cite[Prop.~2.1]{djl}) using induction on $\Card(\cT_X)$.  Choose a toric hyperplane $H\in \cT_X$.  Let $\cT'=\cT_X-\{H\}$ and let $\cT''$ be the restriction of $\cT_X$ to $H$, i.e., $\cT''=\{H\cap H'\mid H'\in \cT_X\}$.   Let $\gS'$ and $\gS''$ denote the singular sets of $\cT'$ and $\cT''$, respectively.
Consider the exact sequence of the triple $(T, \Sigma, \Sigma')$,
	\begin{equation}\label{e:seq1}
	\to H_*(T, \Sigma') \to H_*(T,\Sigma) \to H_{*-1}(\Sigma,\Sigma')\to
	\end{equation}
There is an excision, $H_{*-1}(\Sigma,\Sigma')\cong H_{*-1} (H, \Sigma'')$.  The rank of $\cT'$ is either $n$ or $n-1$, while the rank of $\cT''$ is always $n-1$.  The argument breaks into two cases depending on the rank of $\cT'$.
\smallskip

\noindent
\emph{Case 1}: the rank of $\cT'$ is $n$.  By induction, $H_*(T,\gS')$ and $H_*(H,H\cap \gS)$ 
are free abelian and concentrated in degrees $n$ and $n-1$, respectively.  So, \eqref{e:seq1} becomes,
	\[
	0 \to H_n(T, \Sigma') \to H_n(T,\Sigma) \to H_{n-1}(H,H\cap\Sigma')\to 0
	\]
and all other terms are $0$.  Therefore, $H_*(T,\gS)$ is concentrated in degree $n$ and $H_n(T,\gS)$ is free abelian.
\smallskip

\noindent
\emph{Case 2}: the rank of $\cT'$ is $n-1$.  Then the projection $T\to \ol{T}$ takes $H$ isomorphically onto $\ol{T}$ and the arrangement $\cT''$ on $H$ maps isomorphically to the arrangement $\ol{\cT}_X$ on $\ol{T}$.  So, $(H, H\cap \gS')\cong (\ol{T}, \ol{\gS})$. By \eqref{e:tsigma}, $(T,\gS')\cong K_X \times (H, H\cap \gS') \cong \C^*\times (H, H\cap \gS')$.  By the K\"unneth Formula, $H_*(T,\gS') \cong 
H_*(\C^*)\otimes H_*(H,H\cap \gS')$. So, 
	\begin{align*}
	H_{n-1}(T,\gS')& \cong H_0(\C^*) \otimes H_{n-1} (H,H\cap \gS') \text{ \ and}\\
	H_n(T,\gS')& \cong H_1(\C^*) \otimes H_{n-1} (H,H\cap \gS');
	\end{align*}
moreover, the first isomorphism is induced by the inclusion $(H,H\cap \gS')\to (T, \gS')$.  So, \eqref{e:seq1} becomes,
	\begin{align*}
	0& \to H_1(\C^*)\otimes H_{n-1}(H, H\cap \gS') \to H_n(T,\Sigma) \to H_{n-1}(H,H\cap\Sigma') \\
	&\to H_0(\C^*)\otimes H_{n-1}(H, H\cap \gS')
	\end{align*}
where the last map is an isomorphism.  It follows that $H_{n-1}(T,\gS)=0$ and that $H_n(T,\gS)\cong H_1(\C^*)\otimes H_{n-1}(H, H\cap \,\gS')$, which, by inductive hypothesis,  is free abelian.  This proves the lemma.
\end{proof}

\paragraph{Complexified toric arrangements}\label{ss:real}
In \cite{Dantdel} D'Antonio-Delucchi consider the case of ``complexified toric arrangements.''  This means that for each $(\chi,a)\in X$,
the complex number $a$ has modulus $1$ (where $X\subset \gL\times \C^*$ is a set defining a toric arrangement $\cT_X$). Let $\tc=(S^1)^n\subset \C^n$ be the compact torus. Then for each $H\in \cT_X$, $H\cap \tc$ is a compact subtorus of $\tc$.  The set of subtori, $\ctc:= \{H\,\cap\, \ct\mid H\in \cT\}$, is called the associated \emph{compact arrangement}.

Let $\gsc:=\gS_X \cap \tc$.  We note that $(T,\gS_X)$ deformation retracts onto $(\tc,\gsc)$.  Here are a few obervations.
\begin{enumeratei}
\item
The universal cover of $\tc$ is $\R^n$ (actually the subspace $i\R^n\subset \C^n$). Let $\pi:\R^n\to \tc$ be the covering projection.   Then for each $\tH\in \ctc$, each component of $\pi^{-1}(\tH)$ is an affine hyperplane and the collection of these hyperplanes is a periodic affine hyperplane arrangement in $\R^n$.
\item
If $\cT_X$ is essential, then $\gsc$ cuts $\tc$ into a disjoint union of convex polytopes, called \emph{chambers} (see  \cite{SimoLuca}) .  
The inverse images of these polytopes under $\pi$ give a tiling of $\R^n$. 
\item
When $\cT_X$ is essential, it follows from (ii) that for $n\ge 2$, $\gsc$ is connected  and that for $n\ge 3$,  $\pi_1(\gsc)=\pi_1(\tc)$.
\item
It is easy to prove Lemma~\ref{l:conc} in the case of a compact arrangement.  We have an excision $H_*(\tc,\gsc)\cong H_*(\coprod (P_i, \partial P_i))$ where each chamber $P_i$ is an $n$-dimensional convex polytope.  Hence, $H_*(\tc,\gsc)$ is concentrated in degree $n$ and is free abelian.  Moreover, the rank of $H_*(\tc,\gsc)$ is the number of chambers.
\item
Let $\tgsc$ denote the inverse image of $\gsc$ in $\rr^n$ and let $\wt{\gS}_X$ be the induced cover of $\gS_X$.  Suppose $\cT_X$ is essential.  Then $\tgsc$ cuts $\rr^n$ into compact  chambers.  It follows that $\tgsc$ (and hence, $\wt{\gS}$) is homotopy equivalent to a wedge of $(n-1)$-spheres.
\end{enumeratei}

\section{Certain covers and their nerves}\label{s:covers}

Equip the torus $T=(\C^*)^n$ with an invariant metric. This lifts to a Euclidean metric on $\C^n$ induced from an inner product.  Hence, geodesics in $T$ lift to straight lines in $\C^n$ and each component of the inverse image of a subtorus of $T$ is an affine subspace of $\C^n$.  A \emph{convex subset} of $T$ means a geodesically convex subset.  Thus, each component of the inverse image of a convex subset of $T$ is a convex subset of $\C^n$.

The intersection of an open convex subset of $T$ with the toric hyperplanes in $\cT_X$ is equivalent to an affine arrangement.  An open convex subset $U\subset T$ is \emph{small} (with respect to $\cT_X$) if this affine arrangement is central.  In other words, $U$ is \emph{small}
if the following two conditions hold (cf.~\cite{djl,djlo}):
\begin{enumeratei}
\item
$\{G\in \ol{L}(\cT_X) \mid G\cap U\neq\emptyset\}$ has a unique minimum element, $\MU$.
\item
A toric hyperplane $H\in\cT_X$ has nonempty intersection with $U$ if and only if $\MU \subset H$.
\end{enumeratei}
If (i) and (ii) hold, then the arrangement in $U$ is equivalent to the tangential arrangement along $\MU$, which we denote by $\ca_{\MU}$. The intersection of two small convex open sets is also a small convex set; hence, the same is true for any finite intersection of such sets.

Let $\cu=\{U_i\}_{i\in I}$ be an open cover of $T$ by small convex sets,  put
\[
\cu_\sg:= \{U\in \cu\mid U\cap \gS_X\neq \emptyset\}.
\]

\bigskip

Given a nonempty subset $\gs\subset I$, put
$U_\gs:=\bigcap_{i\in \gs} U_i$.  The \emph{nerve} $N(\cu)$ of
$\cu$ is the simplicial complex  defined as follows.  Its
vertex set is $I$ and a finite, nonempty subset $\gs\subset I$ spans a
simplex of $N(\cu)$ if and only if $U_\gs$ is nonempty. We have the following lemma.

\begin{lemma}\label{l:nu}
Suppose $\Cal T_X$ is essential.
$N(\cu)$ is homotopy equivalent to $T$ and $N(\cu_\sg)$ is a
subcomplex homotopy equivalent to $\gS_X$.  Moreover, $H_*(N(\cu),N(\cu_\sg))$ is concentrated in degree $n$ and $H_n(N(\cu),N(\cu_\sg))$ is free abelian.
\end{lemma}

\begin{proof}
$\cu_\sg$ is an open cover of a neighborhood of $\gS_X$ which deformation retracts onto 
$\gS_X$.  For each simplex $\gs$ of $N(\cu)$, $U_\gs$ is contractible (in fact, it is a small convex open set).
By a well-known result (see \cite[Cor.\ 4G.3 and Ex.\ 4G(4)]{hatcher}) $N(\cu)$ is homotopy equivalent to $T$ and $N(\cu_\sg)$ is homotopy equivalent to $\gS_X$.  The last sentence of the lemma  follows from Lemma~\ref{l:conc}.
\end{proof}

\begin{definition}\label{d:beta}
$\gb(\Cal T_X)$ is the rank of $H_n(N(\cu), N(\cu_\sg))$.
\end{definition}

Equivalently, $\gb(\Cal T_X)$ is the rank of $H_n(T,\gS_X)$. It is not difficult to see that, for essential arrangements, 
$(-1)^n\gb(\Cal T_X)=e(T,\gS_X)=-e(\gS_X)=e(\Cal R_X)$, where $e(\,)$ denotes Euler characteristic.

\section{Local coefficients}\label{ss:local}

\paragraph{Generic and nonresonant coefficients}
Consider an affine arrangement $\ca$
The fundamental group $\pi$ of its complement,  $M(\ca)$, is generated by loops $a_H$ for $H \in \ca$, where the loop $a_H$ goes once around the hyperplane $H$ in the ``positive'' direction.  Let $\ga_H$ denote the image of $a_H$ in
$H_1(M(\ca))$.  Then $H_1(M(\ca))$ is free abelian with basis
$\{\ga_H\}_{H\in \ca}$.  So, a homomorphism $H_1(M(\ca))\to \cc^*$ is
determined by an $\ca$-tuple $\gL\in (\cc^*)^\ca$, where
$\gL=(\gl_H)_{H\in \ca}$ corresponds to the homomorphism sending
$\ga_H$ to $\gl_H$.  Let $\psi_\gL:\pi\to \cc^*$ be the composition of
this homomorphism with the abelianization map $\pi\to H_1(M(\ca))$.
The resulting rank one local coefficient system on $M(\ca)$ is denoted $A_\gL$.

Returning to the case where $\Cal T_X$ is a toric arrangement, for each
simplex $\gs$ in $N(\wu)$, let $\ca_\gs:=\ca_{\Min(U_\gs)}$ be the corresponding central
arrangement (so that $\wh{U}_\gs \cong M(\ca_{\gs})$). 
Given $\gL_\gs \in (\cc^*)^{\Cal A_\gs}$, put
\[
\gl_\gs:=\prod_{H\in \ca_\gs} \gl_H.
\]
Let $A_{\gL_T}\in \Hom (H_1(\Cal R_X),\cc^*)$ 
be a  local coefficient system on $\Cal R_X$.  
The localization of $A_{\gL_T}$ on the open set $\wh{U}_\gs$  has the form $A_{\gL_\gs}$, where $\gL_\gs$ is a $\Cal A_\gs$-tuple in $\cc^*$. We call $\gL_T$ \emph{generic} if $\gl_\gs \neq 1$ for all $\gs\in N(\cu_\sg)$.
We call $\gL_T$ \emph{nonresonant} if $\gL_\gs$ is nonresonant in the sense of \cite{CoOr} for all $\gs\in N(\cu_\sg)$ i.e.,  if the Betti numbers of $M(\Cal A_\gs)$ with coefficients in $A_{\gL_\gs}$ are minimal.

\paragraph{$\eltwo$-cohomology and coeffiicients in a group von Neumann algebra}
For a discrete group $\pi$, $\eltwo\pi$ denotes the Hilbert space of
complex-valued, square integrable functions on $\pi$.  There are
unitary $\pi$-actions on $\eltwo\pi$ by either left or right
multiplication; hence, $\cc\pi$ acts either from the left or right as
an algebra of operators.  The \emph{associated von Neumann algebra}
$\cn\pi$ is the commutant of $\cc \pi$ (acting from, say, the right on
$\eltwo\pi$).

Given a finite CW complex $Y$ with fundamental group $\pi$, the space
of $\eltwo$-cochains on the universal cover $\wt{Y}$ is equal to
$C^*(Y;\eltwo\pi)$, the cochains with local coefficients in $\eltwo
\pi$.  The image of the coboundary map need not be closed; hence,
$H^*(Y;\eltwo\pi)$ need not be a Hilbert space.  To remedy this, one
defines the \emph{reduced} $\eltwo$-cohomology $\hred^*(Y;\eltwo\pi)$  to be
the quotient of the space of cocycles by the closure of the space of
coboundaries.  
The von Neumann algebra admits a trace.  Using this, one can attach a
``dimension,'' $\dim_{\cn\pi} V$, to any closed, $\pi$-stable subspace
$V$ of a finite direct sum of copies of $\eltwo\pi$ (it is the trace
of orthogonal projection onto $V$).  The nonnegative real number
$\dim_{\cn\pi}(\hred^p(Y;\eltwo\pi))$ is the $p^{\mathrm {th}}$
\emph{$\eltwo$-Betti number} of~$Y$.

A technical advance of L\"uck \cite[Ch.~6]{luckbk} is the use local
coefficients in $\cn\pi$ in place of the previous version of
$\eltwo$-cohomology.   
He shows there is a well-defined
dimension function on $\cn\pi$-modules, $A\to \dim_{\cn\pi} A$, which gives the same answer
for $\eltwo$-Betti numbers, i.e., for each~$p$ 
one has that 
$\dim_{\cn\pi}H^p(Y;\cn\pi)=\dim_{\cn\pi}\hred^p(Y;\eltwo\pi)$.  

\paragraph{Group ring coeffiicients} Let $Y$ be a connected CW complex, $\pi=\pi_1(Y)$ and $r:\wt{Y}\to Y$ the universal
cover. There is a well-defined action of $\pi$ on $\wt{Y}$ and hence, on the cellular chain complex of $\wt{Y}$. Given the left $\pi$-module $\zz\pi$, define the cochain complex with \emph{group ring coefficients}
\[
C^*(Y;\zz\pi):=\Hom_\pi(C_*(\wt{Y}),\zz\pi).
\]
Taking cohomology gives $H^*(Y;\zz\pi)$.  

\section{The Mayer-Vietoris spectral sequence}\label{s:mv}

\paragraph{Statements of the main theorems}
Suppose $\cT_X$ is an essential toric arrangement in $T$ and $\pi=\pi_1(\mt)$. 

\begin{theorem}\label{t:generic} Let $\gL_T$ be a generic $X$-tuple with entries in $k^*$.  Then $H^*(\Cal R_X;A_{\gL_T})$ is concentrated in
degree $n$ and $$\dim_k H^n(\mt;A_{\gL_T})=\gb(\Cal T_X).$$
\end{theorem}

\begin{theorem}\label{t:eltwo}
\textup{(cf.~\cite{dl})}.
The $\eltwo$-Betti numbers of $\mt$ are $0$ except in degree $n$ and $\eltwo b_n (\mt)=\gb(\cT_X)$.
\end{theorem}

\begin{theorem}\label{t:gprg}
\textup{(cf.~\cite{djlo,do})}.
$H^*(\mt;\zz\pi)$ vanishes except in degree $n$ and $H^n(\mt;\zz\pi)$ is free abelian.
\end{theorem}

\begin{remark}
Suppose $W$ is a Euclidean reflection group acting on $\rr^n$ and that $\zz^n \subset W$ is the subgroup of translations.  The quotient $W':= W/\zz^n$ is a finite Coxeter group.  The reflection group $W$ acts on the complexification $\cc^n$ and $W'$ acts on the torus $T=\cc^n/\zz^n$.  The image of the affine reflection arrangement in $\cc^n$ gives a toric arrangement $\cT_X$ in $T$.  The fundamental group of $\mt$ is the Artin group $A$ associated to $W$ and $\mt$ is the Salvetti complex associated to $A$.  The quotient of the compact torus by $W'$ can be identified with the fundamental simplex $\gD$ of $W$ on $\rr^n$.  (If $W$ is irreducible, then $\gD$ is a simplex.)  It follows that $\gb(\cT_x)$ is the order of $W'$ (i.e., the index of $\zz^n$ in $W$).  So, in this case Theorem~\ref{t:eltwo} is a special case of the main result of \cite{dl} and Theorem~\ref{t:gprg} is a special case of a result of \cite[Thm.~4.1]{do}.
\end{remark}

\begin{lemma}\label{l:central}
Suppose $\ca$ is a finite, central arrangement of affine hyperplanes.  Let $\pi'=\pi_1(M(\ca))$.  Then 
\begin{enumeratei}
\item \textup{(cf.~\cite{STV, CoOr, djl})}.
For any generic system of local coefficients $A$, 
$H^*(M(\ca);A)$ vanishes in all degrees.
\item \textup{(cf.~\cite{djl})}.
$H^*(M(\ca);\cn \pi')$ vanishes in all degrees.  Hence, all $\eltwo$-Betti numbers are $0$.
\item \textup{(cf.~\cite{djlo})}.
If the rank of $\ca$ is $l$, then $H^*(M(\ca); \zz\pi')$ vanishes except in the top degree, $l$.
\end{enumeratei}
\end{lemma}

\paragraph{Proofs using the Mayer-Vietoris spectral sequence}
The proofs of these three theorems closely follow the argument in \cite{dl}, \cite{djl} and particularly, in \cite{djlo}.  For $\pi=\pi_1(\mt)$, let $A$ denote one of the left $\pi$-modules  in Section~\ref{ss:local}.

Let $\cu=\{U_i\}$ be an open cover of $T$ by small convex sets.
We may suppose that $\cu$ is finite and that it is closed under taking intersections.
For each $G\in \ol{L}_X$, put 
\begin{align*}
	\cu_G:\!\!&=\{U\in \cu\mid \MU\le G\},\\
	\cu^\sg_G:\!\!&=\{U\in \cu\mid \MU< G\}=\{U\in
	\cu_G\mid U\cap \gS_X \cap G\neq\emptyset\}.\\
\end{align*}
The open cover $\cu$ restricts to an open cover $\widehat{\cu}=\{U-\gS_X\}_{U\in \cu}$ of $\rx$.
Any element $\widehat{U}=U-\gS_X$ of the cover is homotopy equivalent to the complement of a central arrangement $M(\ca_{\MU})$.

Suppose $N(\cu)$ is the nerve of $\cu$ and $N(\cu_G)$ is the subcomplex defined by $\cu_G$.
Since $N(\cu_G)$ and $N(\cu^\sg_G)$ are nerves of covers of $G$ and $\gS_X\cap G$, respectively, by contractible open subsets, we have that for each $G\in \OL(\ca)$, 
\begin{equation}\label{e:nerve}
	H^*(N(\cu_G),N(\cu^\sg_G))=H^*(G,\gS(\cT_X\cap G)).
\end{equation}
For each $k$-simplex $\gs=\{i_0,\dots, i_k\}$ in $N(\cu)$, let 
\[
	U_\gs:=U_{i_0}\cap \cdots \cap U_{i_k} 
\]
denote the corresponding intersection.

Let $r:\wt{\Cal R}_X \to \rx$ be the universal cover.
The induced open cover $\{r^{-1}(\widehat{U})\}$ of $\wt{\Cal R}_X$ has the same nerve $N(\widehat{\cu})$ ($=N(\cu)$).
We have the Mayer--Vietoris double complex, 
\[
	C_{i,j}:=\bigoplus_{\gs\in N^{(i)}} C_j(r^{-1} (\widehat{U}_\gs)), 
\]
where $N^{(i)}$ denotes the set of $i$-simplices in $N(\cu)$ (cf.~\cite[Ch.~VII]{brown}.) We get a corresponding double cochain complex, 
\begin{equation}\label{e:e0}
	E^{i,j}_0:= \Hom_\pi(C_{i,j},A), 
\end{equation}
where $\pi= \pi_1(\rx)$.
The filtration on the double complex gives a spectral sequence converging to the associated graded module for cohomology: 
\[
	\GR H^m(\rx;A)=E_\infty := \bigoplus_{i+j=m} E^{i,j}_\infty . 
\]

By first using the horizontal differential, there is a spectral sequence with $E_1$ page
\[
E^{i,j}_1=C^i(N(\cu);\ch^j(A) )
\]
where $\ch^j(A)$ is the coefficient system on $N(\cu)$ defined by 
	\[
		\gs\mapsto H^j(\wu_\gs; A), 
	\]
where $\wu_\gs\cong M(\ca_{\Min(U_\gs)})$.  
For $A=A_{\gL_T}$ or $A=\cn\pi$ these coefficients are $0$ for $G\neq T$.  For $A=\zz\pi$, they are $0$ for $j\neq \dim(G)$.  Hence, in all cases, for any coface $\gs'$ of $\gs$, if $G':=\Min(U_{\gs'}) < G$, the coefficient homomorphism $H^j(M(\ca_G);A)\to H^j(M(\ca_{G'});A)$ is the zero map.   
Moreover,  the $E_1$ page of the spectral sequence decomposes as a direct sum (cf.~\cite[Lemma~2.2]{do}).
In fact, for a fixed $j$, by using Lemma~\ref{l:central}, we see that the $E^{i,j}_1$ term decomposes as 
	\begin{equation*}
		E^{i,j}_1=\bigoplus_{G\in \OL_X^{n-j}} C^i(N(\cu_G),N(\cu^\sg_G);H^j(M(\ca_G);A)), 
	\end{equation*}
where we have constant coefficients in each summand. Hence, at $E_2$ we have 
	\begin{align}
		E^{i,j}_2&=\bigoplus_{G\in \OL_X^{n-j}} H^i(N(\cu_G),N(\cu^\sg_G);H^j(M(\ca_G);A))\notag\\
		&=\bigoplus_{G\in \OL_X^{n-j}} H^i(G,\gS_X\cap G;H^j(M(\ca_G);A)), \label{e:big}
	\end{align}
where the second equation follows from \eqref{e:nerve}.   

When $A=A_{\gL_T}$ or $A=\cn\pi$, all summands vanish for $G\neq T$ and $j\neq 0$.  So, we are left with $E^{n,0}_2 =H^n(T, \gS_X; A)$, which is isomorphic to the tensor product free abelian group of rank $\gb(\cT_X)$ with $A$.  It follows that $H^*(\mt;A)$ is concentrated  in degree $n$ and that $\dim_\cc H^n(\mt;A_{\gL_T}) = \gb(\cT_X)= \dim_{\cn\pi}H^n(\mt;\cn\pi)$.  This proves Theorems~\ref{t:generic} and \ref{t:eltwo}.

Consider formula \eqref{e:big} for $A=\zz\pi$.  By Lemma~\ref{l:conc}, $H^i(G,\gS_X\cap G)$ is concentrated in degree $\dim G=n-j$.   Hence, $E^{i,j}_2$ is nonzero (and free abelian) only for $i+j=n$.  It follows that  the spectral sequence degenerates at $E_2$, i.e., $E_2=E_\infty$.  This proves Theorem~\ref{t:gprg}.

\begin{remark}Let us remark that the statement of Theorem \ref{t:generic} holds even if the local system $\gL_T$ is nonresonant or if it verifies the Schechtman, Terao and Varchenko nonresonance conditions in all small open convex sets , i.e. $\gL_\gs$ verifies the nonresonance conditions in \cite{STV} for all $\gs\in N(\cu_\sg)$. Indeed under these conditions Lemma \ref{l:central} holds.
\end{remark}

\obeylines
M. W. Davis, Department of Mathematics, The Ohio State University, 231 W. 18th Ave., Columbus Ohio 43210  {\tt mdavis@math.ohio-state.edu} 
S. Settepanella, Scuola Superiore Sant'Anna, Piazza Martiri 2,  56127 Pisa, Italy {\tt s.settepanella@sssup.it}

\end{document}